\documentclass[a4paper]{article}
\usepackage{amsthm,amssymb,amsmath,enumerate,graphicx,epsf}

\newcommand{\COLORON}{0}
\newcommand{\NOTESON}{0}
\newcommand{\Debug}{0}
\usepackage{bbold}
\usepackage[usenames]{color} 
\usepackage{amsthm,amssymb,amsmath,bbm,enumerate,graphicx,epsf,stmaryrd}
\usepackage[bookmarks, colorlinks=false, breaklinks=true]{hyperref} 

\newcommand{\comment}[1]{}
\newcommand{\COMMENT}[1]{}

\definecolor{darkgray}{rgb}{0.3,0.3,0.3}
\newcommand{\defi}[1]{{\color{darkgray}\emph{#1}}}

\newcommand{\acknowledgement}{\section*{Acknowledgement}}

\comment{
	\begin{lemma}\label{}	
\end{lemma}
\begin{proof}

\end{proof}

\begin{theorem}\label{}
\end{theorem} 
\begin{proof} 	

\end{proof}

}

\newtheorem{proposition}{Proposition}[section]
\newtheorem{definition}[proposition]{Definition}
\newtheorem{theorem}[proposition]{Theorem}
\newtheorem{corollary}[proposition]{Corollary}

\newtheorem{lemma}[proposition]{Lemma}

\newtheorem{conjecture}{{Conjecture}}[section]

\newtheorem{problem}[conjecture]{{Problem}}

\newtheorem{examp}[proposition]{Example}
\newtheorem{claim}{Claim}

\newcommand{\FIG}{0}

\ifnum \NOTESON = 1 \newcommand{\note}[1]{ 

\hspace*{-30pt}
	{\color{blue}  NOTE: \color{Turquoise}{\small  \tt \begin{minipage}[c]{1.1\textwidth}  #1 \end{minipage} \ignorespacesafterend }} 
	
	}
\else \newcommand{\note}[1]{} \fi

\newcommand{\afsubm}[1]{ \ifnum \Debug = 1 {\mymargin{#1}}
\fi} 

\ifnum \Debug = 1 
\else  \fi

\ifnum \FIG = 1 \newcommand{\fig}[1]{Figure ``{#1}''}
\else \newcommand{\fig}[1]{Figure~\ref{#1}} \fi

\ifnum \FIG = 1 
\else  \fi

\ifnum \Debug = 1 \usepackage[notref,notcite]{showkeys}
\fi

\ifnum \COLORON = 0 \renewcommand{\color}[1]{}
\fi

\newcommand{\R}{\ensuremath{\mathbb R}}

\newcommand{\Z}{\ensuremath{\mathbb Z}}

\newcommand{\BS}{\ensuremath{\mathbb S}}

\newcommand{\cf}{\ensuremath{\mathcal F}}

\newcommand{\sm}{\backslash}

\makeatletter
\DeclareRobustCommand{\cev}[1]{%
  \mathpalette\do@cev{#1}%
}
\newcommand{\do@cev}[2]{%
  \fix@cev{#1}{+}%
  \reflectbox{$\m@th#1\vec{\reflectbox{$\fix@cev{#1}{-}\m@th#1#2\fix@cev{#1}{+}$}}$}%
  \fix@cev{#1}{-}%
}
\newcommand{\fix@cev}[2]{%
  \ifx#1\displaystyle
    \mkern#23mu
  \else
    \ifx#1\textstyle
      \mkern#23mu
    \else
      \ifx#1\scriptstyle
        \mkern#22mu
      \else
        \mkern#22mu
      \fi
    \fi
  \fi
}

\makeatother

\newcommand{\g}{\ensuremath{G\ }}

\newcommand{\Lr}[1]{Lemma~\ref{#1}}

\newcommand{\Tr}[1]{Theorem~\ref{#1}}
\newcommand{\Trs}[1]{Theorems~\ref{#1}}
\newcommand{\Sr}[1]{Section~\ref{#1}}

\newcommand{\Prr}[1]{Pro\-position~\ref{#1}}

\newcommand{\Dr}[1]{De\-fi\-nition~\ref{#1}}

\newcommand{\Cg}{Cayley graph}

\renewcommand{\iff}{if and only if}
\newcommand{\fe}{for every}

\newcommand{\st}{such that}

\newcommand{\ti}{there is}

\newcommand{\obda}{without loss of generality}

\newcommand{\wrt}{with respect to}

\newcommand{\mymargin}[1]{
 \ifnum \Debug = 1
  \marginpar{%
    \begin{minipage}{\marginparwidth}\small%
      \begin{flushleft}%
        {\color{blue}#1}%
      \end{flushleft}%
   \end{minipage}%
  }%
 \fi
}%

\newcommand{\mySection}[2]{}

\usepackage{mathtools,enumitem}
\usepackage[percent]{overpic}

\usepackage{authblk}
\usepackage{hyperref}
\RequirePackage{latexsym} \RequirePackage{amsthm}
\RequirePackage{amsmath}

\usepackage{enumerate}
\RequirePackage{amssymb} \RequirePackage{makeidx}
\usepackage{dsfont}
\usepackage{mathrsfs}

\newcommand{\sh}{\mathrm{Shadow}}
\renewcommand{\COMMENT}[1]{}

\usepackage{authblk}

\begin{document}

\title{$2$-complexes with unique embeddings in 3-space}

\author[1]{Agelos Georgakopoulos\thanks{Supported by the European Research Council (ERC) under the European Union's Horizon 2020 research and innovation programme (grant agreement No 639046), and by EPSRC grants EP/V009044/1 and EP/V048821/1.}}

\author[2]{Jaehoon Kim\thanks{Supported by the Leverhulme Trust Early Career Fellowship ECF-2018-538, by the POSCO Science Fellowship of POSCO TJ Park Foundation, and by the KAIX Challenge program of KAIST Advanced Institute for Science-X.}}
\affil[1]{{Mathematics Institute}\\
 {University of Warwick}\\
  {CV4 7AL, UK}}
\affil[2]{Department of Mathematical Sciences, KAIST, South Korea 34141}

\date{\today}
\maketitle

\begin{abstract}
A well-known theorem of Whitney states that a 3-connected planar graph admits an essentially unique embedding into the 2-sphere. We prove a 3-dimensional analogue:  a simply-connected $2$-complex every link graph of which is 3-connected admits an essentially unique locally flat embedding into the 3-sphere, if it admits one at all. This can be thought of as a generalisation of the 3-dimensional Schoenflies theorem.
\end{abstract}



\section{Introduction}

The following classical theorem of Whitney states that a $3$-connected planar graph admits an essentially unique embedding into the plane or the 2-sphere. A graph is \defi{$k$-connected}, if it  remains connected after removing any set of $k-1$ vertices.

\begin{theorem}[{Whitney's theorem  \cite[Theorem~11]{whitney_congruent_1932}, \cite[Theorem~4.3.2]{diestelBook05}}] \label{Whitney}

Let\\ \g  be a 3-connected, finite, planar graph. Then \fe\ two embeddings $\phi,\psi: G \to \BS^2$, \ti\ a homeomorphism $\alpha: \BS^2 \to \BS^2$ \st\ $\psi= \alpha \circ \phi$. 
\end{theorem}

This can be thought of as a generalisation of the Schoenflies theorem, which can be phrased as saying that any two embeddings of $\BS^1$ in $\BS^2$ differ by an automorphism of $\BS^2$. The Schoenflies theorem has been generalised to higher dimensions \cite{BrownGS, MazurGS}, and its 3-dimensional version reads as follows

\begin{theorem}[{Generalised Schoenflies theorem \cite{BrownGS, MazurGS}}] \label{GST}
If $K\subset \BS^3$ is homeomorphic to $\BS^2$ and locally flat, then there is an automorphism of $\BS^3$ mapping its equator (which is homeomorphic to $\BS^2$) to $K$. 
\end{theorem}
The local flatness condition here is a weakening of piecewise linearity; see \Sr{loc flat}. Its necessity became clear after Alexander famously introduced his horned sphere, a subset of $\BS^3$  homeomorphic to $\BS^2$, the exterior of which is not simply connected \cite{Alexander}, \cite[p.~169]{Hatcher}.

\medskip
In this paper we prove a 3-dimensional version of Whitney's \Tr{Whitney}, generalising \Tr{GST} (which is used in our proof):

\begin{theorem} \label{Whit3D intro}
Let $X$ be a finite, simply-connected, $2$-complex every link graph of which is 3-connected. Then, for every  two locally flat embeddings $\phi,\psi: X \to \BS^3$, there exists a homeomorphism $\alpha : \BS^3 \to \BS^3$ such that $\psi= \alpha \circ \phi$.
\end{theorem}

Our notion of local flatness for $2$-complexes generalises the standard one for 2-manifolds, and is introduced in \Sr{loc flat}. The reader will lose nothing by replacing it with piecewise linearity. 

In fact we will prove a stronger statement than  \Tr{Whit3D intro}, by relaxing the 3-connectedness requirement on the link graphs to a much weaker one. See \Sr{sec Wh comp} for details. 

Simple connectedness is clearly an indispensable condition in \Tr{Whit3D intro}: if $X$ is a triangulation of a torus, for example, then we can embed it in $\BS^3$ in at least two ways, either knotted or unknotted. This example can be easily modified to make all link graphs 3-connected, by adding 2-cells inside the torus. In \Sr{sec Wh comp} we also provide further examples showing that our main result is tight in the sense that none of its conditions can be dropped.

The finiteness of $X$ is also easily seen to be necessary. For example, let $X$ be the standard Cayley complex of $\Z^3$. Its link graphs are all isomorphic to the 1-skeleton of the octahedron, which is $3$-connected. But there are various ways to embed it into $\BS^3$: we could embed it with a unique accumulation point of images of vertices, or we could embed it inside half of $\BS^3$, with every point of the equator being an accumulation point. Yet, we can allow $X$ to be infinite in  \Tr{Whit3D intro} by imposing additional conditions on the embeddings: 

\begin{theorem} \label{Whit3D inf intro}
Let $X$ be a locally finite, simply-connected, Whitney, $2$-complex, every link graph of which is 3-connected. Let $\phi,\psi: X \to \R^3$  be locally flat, accumulation-free embeddings with bounded chambers. Then there exists a homeomorphism $\alpha : \R^3 \to \R^3$ such that $\psi= \alpha \circ \phi$.
\end{theorem}



\medskip

Whitney's \Tr{Whitney} has many applications, one of which is understanding group actions on $\BS^2$ or other 2-manifolds. For example, it is known that a finite group $\Gamma$ admits a faithful action by homeomorphisms on $\BS^2$ \iff\ $\Gamma$ has a planar \Cg. To prove the backward direction of this statement, one can use \Tr{Whitney} to extend the canonical action of $\Gamma$ on a \Cg\ embedded in $\BS^2$ to an action on $\BS^2$. See \cite{Kleinian} for more. This paper was motivated by the quest (currently in progress) to extend such results to groups acting on $\BS^3$ or other 3-manifolds. In parallel work in progress, George Kontogeorgiou (private communication) has used \Tr{Whit3D intro} to deduce that if a finite group $\Gamma$ has a Cayley complex embeddable in $\BS^3$, then $\Gamma$ admits a faithful action by isometries on $\BS^3$.

In a recent series of papers, Carmesin \cite{CarEmbItoV} obtained a characterization of the 2-complexes admitting an embedding in $\BS^3$ in terms of forbidden minors, in the spirit of Kuratowski's characterization of the planar graphs. Carmesin's work provided additional motivation, and tools, for the current work. More generally, there is currently an interest for topological extensions of graph-theoretic results to 2-complexes, see e.g.\ \cite{CarMihOut,spheres,KLNSUni}, and this paper adds to this trend.

\medskip

Many of the concepts and results of this paper can be easily extended to higher dimensions, but for others, in particular the results of \Sr{sec chambers}, this seems more difficult. We would be interested to see extensions of \Tr{Whit3D intro} to higher dimensions:
\begin{problem}
Let $X$ be a finite, $n$-dimensional cell complex. Under what conditions is it true that for every  two locally flat embeddings $\phi,\psi: X \to \BS^{n+1}$, there exists a homeomorphism $\alpha : \BS^{n+1} \to \BS^{n+1}$ such that $\psi= \alpha \circ \phi$?
\end{problem}
For example, it is natural to assume that $X$ is $(n-1)$-connected, and its links satisfy a recursive condition ensuring their unique embeddability in $\BS^{n}$.

\medskip
The proof of \Tr{Whit3D intro} can be summarized as follows. The images $\phi(X), \psi(X)$ of $X$ separate $\BS^3$ into chambers, and we want to apply \Tr{GST} on each chamber of $\phi$ to map it to a chamber of $\psi$, after which we can combine these maps with $\psi^{-1} \circ \phi: X \to X$ to obtain $\alpha$. In order to carry out this plan, we need that A) each chamber is bounded by a copy of $\BS^2$, and B) the subcomplexes of $X$ bounding chambers of $\phi$ or $\psi$  coincide. To check B), we show that  $\phi, \psi$ give rise to the same rotation system on $X$, i.e.\ the same cyclic ordering of 2-cells around each 1-cell of $X$ (\Sr{sec Urs}). Rotation systems  give rise to \defi{local surfaces}, introduced by Carmesin \cite{CarEmbItoV} as a combinatorial description of boundaries of chambers (\Sr{sec rot sys}), which are helpful in establishing (B). It is possible however that (A) is violated, e.g.\ because the boundary of some chamber contains two copies of $\R^2$ intersecting at a point or line segment. To amend this issue, we fatten $X$ into a larger complex ${\rm fat}(X)$, and show that ${\rm fat}(X)$ satisfies (A) (\Sr{sec fat}). \Sr{sec proof} gives a more accurate implementation of this plan.

\section{Preliminaries}

\subsection{Graphs} \label{graphs}

A (simple) \defi{graph} \g is a pair $(V,E)$ of sets, where $V$ is called the set of \defi{vertex} and $E$ is a set of two-element subsets of $V$, called the set of \defi{edges}. We will write $uv$ instead of $\{u,v\}$ to denote an edge. A \defi{multi-graph} is defined similarly, except $E$ is a multi-set, and it can have elements consisting of just one vertex. 

We let $V(G)$ denote the set of vertices of a graph $G$ and $E(G)$ denote the set of edges of $G$.
 
The \defi{degree} of a vertex $v\in V(G)$ in \g is the number of edges in $E(G)$ containing $v$.

Every (multi-)graph $\g=(V,E)$ gives rise to an 1-complex, by letting $V$ be the set of 0-cells, and for each $uv\in E$ introducing an arc with its endpoints identified with $u$ and $v$. Thus we will sometimes interpret the word \defi{graph} as  an 1-complex. In particular, when discussing \defi{embeddings} of graphs in $\BS^2$ we will mean topological embeddings of 1-complexes.

A \defi{subdivision} of $\g=(V,E)$ is obtained by replacing an edge $uv$ by a pair of edges $uw,vw$, where $w\not\in V$ is a new vertex. Note that subdivisions results in homeomorphic 1-complexes.

\subsection{2-complexes} \label{ccs}

A \emph{$2$-complex} is a topological space $X$ obtained as follows. We start with an 1-complex $X^1$ as defined above, called the \defi{1-skeleton} of $X$. We then introduce a set $X^2$ of copies of the closed unit disc $\mathbb{D} \subseteq \R^2$,  called the \defi{2-cells} or \defi{faces} of $X$, and for each $f\in X^2$ we \defi{attach} $f$ to $X^1$ via a map $\phi_f: \BS^1 \to X^1$, where we think of $\BS^1$ as the boundary of $\mathbb{D}$. Attaching here means that we consider the quotient space where each point $x$ of $\BS^1\subset f$ is identified with $\phi_f(x)$. We let $X^0:= V(X^1)$ be the set of \defi{vertices}, or 0-cells, of $X$.

We say that $X$ is \defi{regular}, if $\phi_f$ is a homeomorphism onto its image \fe\ $f\in X^2$.

For $f\in X^2$, we write $f =[x_1,\dots, x_k]$ if $x_1,\dots, x_k$ is the cyclic sequence of vertices appearing in the image of $\phi_f$. We only use this notation for regular $X$. We say that $u,v,w$ are \emph{consecutive} vertices in $f$, if $u,v,w$ is a consecutive subsequence of a sequence  $x_1,\dots, x_k$ as above.

To each face $f=[x_1,\dots, x_k] \in X^2$, we associate two distinct \defi{directed faces} 
$f_1,f_2$, also denoted by $\langle x_1,\dots, x_t\rangle$ and $\langle x_t,\dots, x_1\rangle$ respectively. Their \defi{reverses} are defined as $f_1^{-1}:= f_2$ and $f_2^{-1}:= f_1$.




We sometimes abuse the notion and consider faces or edges as sets and write $f\cap g$ to denote the set of vertices belonging to both faces. If the set of vertices belonging to both faces forms either a vertex or an edge, we sometimes write $f\cap g$ to denote the vertex or the edge. For example, if $f=[x_1,x_2,x_3]$ and $g=[x_2,x_3,y]$, then $f\cap g = x_2x_3 \in X^1$.

\smallskip
If $X$ is not regular then it is always possible to produce a regular complex $X'$ homeomorphic to $X$ using the \emph{barycentric subdivision} defined as follows. 
For each edge $e=uv\in X^1$, we subdivide $e$ by adding a new vertex $m$ at its midpoint. For each occurrence of $e$ in a 2-cell $f$ of $X$, we replace that occurence by the pair $um, mv$ or $vm, mu$ as appropriate. We then triangulate each face $h=x_1,\ldots,x_k$ of the resulting $2$-complex by adding a new vertex $c$ in its interior, adding the edges $c x_1,\ldots, cx_k$ to the 1-skeleton, and replacing $h$ by the 2-cells $[c, x_1, x_2], [c, x_2, x_3], \ldots, [c,x_k, x_1]$. 

Note that the barycentric subdivision $X'$ of $X$ is a simplicial complex, in particular a regular one, and its 1-skeleton is a simple graph. 

\smallskip
For each $v\in X^0$, the \emph{link graph} $L_X(v)$ is the graph on the neighbourhood of $v$ in $X^1$ with $uw \in E(L_X(v))$ if and only  if $u,v,w$ are consecutive in a face of $X$. Alternatively, we can define $L_X(v)$ so that its vertices are the edges incident with $v$, and two edges $vu,vw$ are joined by an edge of $L_X(v)$ whenever $u,v,w$ are consecutive in a face of $X$. These two definitions yield isomorphic graphs when $X^1$ is a simple graph, and it is a matter of convenience to use the one or the other. If $X$ is not regular then link graphs are more naturally defined as multigraphs, but we will restrict their use to regular complexes, with simple 1-skeleton, in this paper.

We say that $X$ is \emph{locally-$k$-connected}, if each of its link graphs is $k$-connected. Here a graph is \defi{$k$-connected}, if it has at least $k$ vertices, and remains connected after removing any set of at most $k-1$ vertices.

\smallskip
A 2-complex $X$ is \defi{locally finite}, if each vertex of $X$ is contained in finitely many edges and faces. {\bf All 2-complexes in this paper are locally finite.}



\subsection{Chambers}
For an embedding $\phi:X\rightarrow \BS^3$ of a $2$-complex $X$, we call each connected component of $\BS^3 \setminus \phi(X)$  a \emph{$\phi$-chamber} or simply a \emph{chamber}. The \defi{boundary} $\partial C$ of a chamber $C$ is the set of points $x \in \phi(X)$ in the closure of $C$ that are not in $C$.

It follows from elementary topological arguments that, when $X$ is finite, every chamber $C$ is open, and if $\partial C$ intersects the interior of a face $f$, then $f\subseteq C$. We deduce
\begin{proposition} \label{chambers}
Let $\phi:X\rightarrow \BS^3$ be an embedding of a finite, locally 1-connected, $2$-complex $X$.  Then $\partial C$ is a union of faces of $X$ \fe\ $\phi$-chamber $C$.
\end{proposition}

An embedding $\phi: X \to \R^3$ is \defi{accumulation-free}, if $\phi(X^0)$ has no accumulation points in $\R^3$. It can be proved by straightforward topological arguments that, under mild conditions, the boundary of each chamber is connected. We provide a proof for completeness.

\begin{proposition} \label{pC conn}
Let $\phi:X\rightarrow M$ be either an embedding of a finite $2$-complex $X$ in $M= \BS^3$, or an accumulation-free embedding of a locally finite $2$-complex $X$ in $M= \R^3$. \mymargin{this is new} Then $\partial C$ is connected. 
\end{proposition}
\begin{proof}
Suppose $S:=\partial C$ is disconnected, and let $K\subset S$ be one of its components. Let $U\supset K$ be an open subset of $\BS^3$ containing $K$, the closure $\overline{U}$ of which avoids $S \sm K$; our conditions on $\phi$ easily imply the existence of such a $U$. Note that $\partial U$ disconnects $U$ from its complement $U^c$ by elementary topological arguments, hence $\partial U$ disconnects $K$ from $S \sm K$. Since $C$ connects $K$ to  $S \sm K$, this implies that $\partial U\cap C\neq \emptyset$. { Since $Z=C^c$ is a connected set containing both $K$ and $S\sm K$, this implies that $\partial U\cap C^c \neq \emptyset$.} Since $\partial C$ disconnects  $C$ from $C^c$, and $\partial U$ is connected, we deduce that $\partial U\cap \partial C\neq \emptyset$. This contradicts the choice of $U$.
\end{proof}

\subsection{Local flatness} \label{loc flat}

Let us recall the standard notion of local flatness for 2-manifolds. 
An embedding $\phi: \BS^2 \to \BS^3$ is \defi{locally flat}, if  for each $x\in \phi(\BS^2)$ there exists a neighbourhood $U_x$ of $x$ \st\ the topological pair $(U_x,U_x\cap \phi(X))$ is homeomorphic to $(\mathbb{R}^3, \mathbb{R}^2)$, by which we mean that there is a homeomorphism from $U_x$ to $\mathbb{R}^3$ mapping $U_x\cap \phi(X)$ to $\mathbb{R}^2 \subset \mathbb{R}^3$. (A \defi{topological pair} $(X,A)$ consists of a topological space $X$ and a subspace $A\subseteq X$.)

We now extend the notion of local flatness to embeddings of $2$-complexes instead of $2$-manifolds. Before doing so we need to introduce some notation.
For a given planar multigraph $G$, consider a planar embedding $\psi$ of $G$ into the boundary of an open unit ball $B_1$ around the origin in $\mathbb{R}^3$. Let $\sh(G)$  
 be the collection of points $y$ lying on a straight line segment between the origin and some point $x\in \psi(G)$, where the origin is included but $x$ is not. 
 
For a face $f$ of $X$, we write ${\rm int}(\phi(f))$ to denote the \defi{interior} of $\phi(f)$, i.e.\ the set of points of $\phi(f)$ which do not belong to $\phi(e)$ for any $e\in X^1$. Here, we consider the face as a topological disk, and any edge as a topological curve.
Similarly, for $e=uv\in X^1$, we write ${\rm int}(\phi(e))$ to denote the \defi{interior} of $\phi(E)$, i.e.\ the set of points of $\phi(e)$ other than $\phi(v)$ and $\phi(u)$.

We say that an embedding $\phi$ of a 2-complex $X$ in  $\R^3$ or $\BS^3$ is \defi{locally flat}, if  for each $x\in \phi(X)$ there exists a neighbourhood $U_x$ of $x$ satisfying the following:
\begin{enumerate}
	\item[(L1)] if $x \in {\rm int}(\phi(f))$ for some face $f\in X^2$, then $(U_x,U_x\cap \phi(X))$ is homeomorphic to $(\mathbb{R}^3, \mathbb{R}^2)$.
	\item[(L2)] If $x \in {\rm int}(\phi(e))$ for some edge $e \in X^1$, then $(U_x,U_x\cap \phi(X))$ is homeomorphic to $(B_1,\sh(G))$ where $G$ is a multigraph on two vertices with a number of edges equal to the number of 2-cells incident with $e$. \mymargin{Replaced "at least one edge" with "with a number of edges equal to the number of 2-cells incident with $e$"; right?}
	\item[(L3)] If $x = \phi(v)$ for some $v\in X^0$, then  $(U_x,U_x\cap \phi(X))$ is homeomorphic to $(B_1,\sh(G))$ where $G=L_X(v)$ is the link graph of $v$.
\end{enumerate}

\section{Definition of Whitney complexes and statement of the main result} \label{sec Wh comp}

In \Tr{Whit3D intro}  we assumed $X$ to be locally 3-connected. In this section we will re-state our main theorem in a way that significantly relaxes this condition. Before doing so, let us go one dimension down, and revisit planar graphs. 
Whitney's \Tr{Whitney} assumes the graph to be 3-connected, but this condition can be relaxed to yield the following characterisation of graphs with essentially unique planar embeddings

\begin{theorem}[Fleischner \cite{FleUni}] \label{un pl gr}
A connected graph $G$ has a unique, up to automorphism of $\BS^2$,  embedding in $\BS^2$, if and only if $G$ is a subdivision of either\\ 
A) a $3$-connected planar simple graph, or\\ 
B) a connected multigraph with at most three edges. 
\end{theorem}

We remark that \Tr{un pl gr} is not needed for the proof of our main result (\Tr{Whit3D} below), it is provided to motivate the statement of the later. The multigraphs that item (B) allows are the following: an isolated vetex, an edge with multiplicity one, two or three, $K_{1,3}$,  and $C_2 \vee K_2$, i.e.\ a path of length two with one edge having multiplicity two.

A \defi{theta graph} is a subdivision of the multi-graph with 2 vertices and 3 edges between them. 

\medskip
We now return to embeddings of $2$-complexes into $\BS^3$. Assuming  3-connectedness of link graphs in \Tr{Whit3D intro} is a natural way to determine the embedding of $X$ locally at each vertex $v$, by applying \Tr{Whitney} to the surface of a small ball centered at $v$. As we have just seen, it is possible to relax the 3-connectedness condition in \Tr{Whitney}, and this will allow us to relax our condition on the link graphs of $X$ accordingly. However we cannot do so too naively, and the following definition answers the question of how much we can relax the 3-connectedness condition.

Let $\mathcal{F}$ be the set of all $2$-connected graphs with an essentially unique planar embedding. By \Tr{un pl gr}, these are exactly the theta graphs and the subdivisions of 3-connected planar graphs. For a given $2$-complex $X$, let $G(X)$ be the subgraph of the $1$-skeleton $X^1$ consisting of edges belonging to at least three faces. 

\begin{definition} \label{def Whit}
We say that a regular $2$-complex $X$ is \emph{Whitney}, if every link graph $L_{X}(v)$ of $X$ is in $\mathcal{F}$ and $G(X)$ is connected. 
\end{definition}

The assumption that $X$ be regular here is imposed for simplicity only. We can extend the definition verbatim to a non-regular $X$, by allowing $L_{X}(v)$ to be a multigraph. Alternatively, we can define $X$ to be Whitney in the general case, if its barycentric subdivision (see \Sr{ccs}) is Whitney. 

Let $Y(X)$ be the set of vertices $u\in X^0$ \st\ $L_X(u)$ has a vertex of degree at least three.
For any edge $uv\in X^1$ in a Whitney complex, if $u\notin Y(X)$, then the link graph $L_X(u)$ is a cycle, so $uv$ belongs to exactly two faces. Hence if $X$ is Whitney, then $G(X)$ is a connected graph with vertex set $Y(X)$, and $Y(X)$ coincides with the set of vertices $u\in X^0$ \st\ $L_X(u)$ is a subdivision of $3$-connected simple graph or a theta graph.  We allow $G(X)$ to be empty, e.g.\ when $X$ is a triangulation of $\BS^2$.

\medskip
We can now state the main result of this paper, which strengthens \Tr{Whit3D intro}:

\begin{theorem} \label{Whit3D}
Let $X$ be a finite, simply-connected, Whitney, $2$-complex, which is embeddable into $\BS^3$. Then, for every  two locally flat embeddings $\phi,\psi: X \to \BS^3$, there exists a homeomorphism $\alpha : \BS^3 \to \BS^3$ such that $\psi= \alpha \circ \phi$.
\end{theorem}


As remarked in the introduction, both the simple connectedness of $X$ and the local flatness conditions are indispensable here. To see that the connectedness of $G(X)$ (in \Dr{def Whit}) cannot be dropped either, consider the disjoint union $X \cup Z$ of any two simplicial 2-complexes $X$ and $Z$ (uniquely) embeddable in $\BS^3$, and  join $X$ to $Z$ by adding a triangulated cylinder $C$ so that one boundary cycle of $C$ is identified with the boundary of a face $f$ of $X$ and the other boundary cycle is identified with the boundary of a face $g$ of $Y(X)$. The resulting 2-complex can be embedded in $\BS^3$ in more than one ways: either with both $X,Z$ embedded outside $C \cup \{f,g\}$, or with one of $X,Z$ inside and the other outside $C \cup \{f,g\}$.

The fact that all graphs we allow in \cf\ are planar is not a restriction: if $X$ admits an embedding $\phi$ in $\BS^3$, then it is easy to see that each of its link graphs $L_X(v)$ is planar by restricting $\phi$ to the boundary of a small ball around $v$. The
2-connectedness of the graphs we allow in \cf\ is also indispensable. To see this, start with a triangulation of $\BS^2$ with about 50 vertices embedded in $\BS^3$, and add a few cells so as to create one more chamber. It is easy to do this in such a way that the boundaries of the three resulting chambers $C_1,C_2,C_3$ have about 10,30, and 60 vertices, respectively. Now pick an edge $e=uv$ lying in the middle of the boundaries of $C_2$ and $C_3$, add one new vertex $w$ to this complex, and add a 2-cell $uvw$. It is possible to embed the resulting complex $X$ with $w$ either inside $C_2$ or inside $C_3$, and so $X$ violates the conclusion of \Tr{Whit3D}. But $X$ is not Whitney, because the links of $u,v,w$ are not in \cf, although they are uniquely embeddable by \Tr{un pl gr} (the link of $w$ is an edge, while the  links of $u,v$ are subdivisions of $C_2 \vee K_2$). It is easy to produce a similar example where the link of $w$ is a $K_{1,3}$. Thus we have to restrict \cf\ to the 2-connected graphs  of \Tr{Whit3D}.

\subsection{Infinite complexes}
Recall the example of the Cayley complex of $\Z^3$ from the introduction, which shows that we cannot just drop the finiteness condition in \Tr{Whit3D}. However, we can do so under additional restrictions on the embeddings. Recall that an embedding $\phi: X \to \R^3$ is \defi{accumulation-free}, if $\phi(X^0)$ has no accumulation points in $\R^3$. Note that when $\phi$ is accumulation-free, then a $\phi$-chamber $C$ is bounded \iff\ $\partial C$ contains finitely many faces of $X$. Here we used \Prr{chambers}, which easily extends to this setup. We now state our variant of \Tr{Whit3D} for infinite complexes, generalising \Tr{Whit3D inf intro}:

\begin{theorem} \label{Whit3D inf}
Let $X$ be a locally finite, simply-connected, Whitney, $2$-complex. Let $\phi,\psi: X \to \R^3$  be locally flat, accumulation-free embeddings with bounded chambers. Then there exists a homeomorphism $\alpha : \R^3 \to \R^3$ such that $\psi= \alpha \circ \phi$.
\end{theorem}

We remark that infinite versions of \Tr{Whitney} have been obtained by Imrich \cite{ImWhi} and Thomassen \& Richter \cite{ThomassenRichter}. The latter was used by the first author \cite{Kleinian} in order to understand discrete group actions on non-compact surfaces. \Tr{Whit3D inf} may find similar applications on discrete group actions on $\R^3$. 

\subsection{Proof of \Tr{un pl gr}}

Fleischner \cite{FleUni} proved \Tr{un pl gr} using somewhat different terminology. 
We provide the following proof that is shorter than Fleischner's. \mymargin{Do you agree?}
\begin{proof}[Proof of \Tr{un pl gr}]
We begin by noticing that whether $G$ has a unique, up to automorphism of $\BS^2$,   embedding in $\BS^2$, is a topological property depending only on the homeomorphism type of $G$. Thus this property is preserved when subdividing edges (or suppressing subdivisions).

It is thus natural to consider the multigraph  $G'$ obtained from $G$ by suppressing vertices of degree $2$, i.e.\ by replacing each maximal path all of whose internal vertices  have degree $2$ by an edge. If $G'$ is $3$-connected, then we are done by Whitney's \Tr{Whitney}.  Otherwise, we consider various cases. If $G'$ is a connected multigraph with at most three edges, it is clear that it has an essentially unique embedding in $\BS^2$. Thus in the following cases we will assume that $G'$ has an essentially unique embedding, and will show that it has at most three edges.

If $G'$ has at most two vertices, and at least 4 parallel edges between them, then  we can embed these edges in different cyclic orderings, and automorphisms of $\BS^2$ cannot map any such cyclic ordering to any other. 

Thus we may assume from now on that $G'$ contains at least three vertices, and a cut-set of size at most 2.


If $G'$ contains both a cut-vertex and a cycle $C$,  let $K$ be the maximal 2-connected subgraph of $G'$  containing $C$. Then $K$ contains a cut-vertex $v$ of $G'$. 
As $K$ is $2$-connected, it must contain a cycle  containing $v$, and we may assume \obda\ that $C\ni v$ is this cycle. Let $A$ be a component of $G'-v$ which is disjoint from $C$, and $B$ be the component intersecting $C$. When we embed $G'$, we can embed $A$ either inside $C$ or outside $C$. If $G'[G'-A] \neq C$, then these two embeddings are not homeomorphic.
If $G'[B\cup \{v\}]=C$, and if $G'[A\cup \{v\}]$ is not just an edge, then $G'[A\cup \{v\}]$ contains a vertex of degree 3 other than $v$. Thus `flipping' $A$ while fixing $B$ gives two non-homeomorphic embeddings of $G'$.
This proves that $G'$ consists of a 2-cycle $C$ attached to an edge, i.e.\ it coincides with $C_2 \vee K_2$. 

If $G'$ has a cut-vertex $v$ but no cycles at all, then all neighbors of $v$ lie in different components of $G'-v$. In this case, $v$ must have degree at most three if $G'$ has an essentially unique embedding, because  otherwise embedding the components of $G-\{v\}$ in two different cyclic ordering where one is not obtained from the other by flipping yields two non-homeomorphic embeddings.  If $G'$ has another cut-vertex $u$ of degree at least three, then let $P$ be  a path from $u$ to $v$. Then $G-P$ has at least four components, and again embedding them in different cyclic ordering yields non-homeomorphic embeddings. It follows that $G'$ must coincide with $K_{1,3}$. 

Finally, if $G'$ has no cut-vertex but a cut set $\{u,v\}$ of size two, then let $A$ be a component of $G'-\{u,v\}$, and let $B:=G'-\{u,v\}-A$.  Find two paths from $u$ to $v$, one path $P$ through $A$ and another path $Q$ through $B$. Thus $P\cup Q$ is a cycle $C$. As all vertices in $G'$ has degree at least three, $G'[A\cup \{u,v\}]$ contains an edge $e$ not in $C$ and $G'[B\cup \{u,v\}]$ also contains an edge $e'$ not in $C$. We can `flip' the embedding of $A$ while keeping $P$ fixed, to obtain two non-homeomorphic embeddings: one where $e$ and $e'$ both lie inside $C$, and another where exactly one of $e,e'$ lies inside $C$. To make this more precise, pick an embedding $\phi: G' \to \BS^2$, and notice that $u,v$ must lie on the boundary of a face of $G'[A\cup \{u,v\}]$ because they are connected by a path in $B$. Therefore, we can find a closed disc $D\subset \BS^2$ containing the image of  $G'[A\cup \{u,v\}]$, with $u,v$ lying on the boundary of $D$. We can reflect $D$ along $P$ to obtain the second embedding. Hence this case is not possible.
\end{proof}

\section{Rotation systems} \label{sec rot sys}

A \defi{rotation system} of a graph $G$ is a family $(\sigma_v)_{v\in V(G)}$ of cyclic orderings of the edges incident with each vertex $v\in V(G)$. Every embedding of $G$ on an orientable surface defines a rotation system, by taking $\sigma_v$ to be the clockwise cyclic ordering in which the edges incident to $v$ appear in the embedding. The rotation system $(\sigma_v)_{v\in V(G)}$ is said to be \defi{planar}, if it can be defined by an embedding of $G$ in the sphere $\BS^2$. 

Let $X$ be a regular $2$-complex, and let $\overleftrightarrow{E}(X)$ denote the set of \defi{directions} of its 1-cells, that is, the set of directed pairs $\overrightarrow{xy}:= \left<x,y \right>$ such that $xy\in X^1$. Thus every 1-cell gives rise to two elements of $\overleftrightarrow{E}(X)$. 
A \defi{rotation system} of  $X$ is a family $(\sigma_e)_{e \in \overleftrightarrow{E}(X)}$ of cyclic orderings $\sigma_e$ of the faces incident with each  $e = \overrightarrow{xy} \in \overleftrightarrow{E}(X)$, such that if $e'=\overrightarrow{yx}$, then $\sigma_{e'}$ is the reverse of $\sigma_e$. A rotation system $(\sigma_e)_{e \in \overleftrightarrow{E}(X)}$ of  $X$ induces a rotation system $\sigma^v$ at each of its link graphs $L_X(v)$ by restricting to the directions of 1-cells emanating from $v$:  \fe\ $u\in V(L_X(v))$ we let $\sigma^v_{u}$ be the cyclic order obtained from 
$\sigma_{\overrightarrow{vu}}$ by replacing each face $f$ appearing in the latter by the unique neighbour $w$ of $u$ in $V(L_X(v))$ such that $w,u,v$ appear consecutively in $f$.

A rotation system of a regular $2$-complex $X$ is \defi{planar}, if it induces a planar  rotation system on each of its link graphs.
Note that every locally flat embedding $\phi$ of $X$ into a $3$-manifold defines a planar rotation system, by letting $\sigma_e$ be the cyclic order in which the images of the faces incident with $e$ appear in $U_x$, where $x$ is any interior point of $\phi(e)$, and $U_x$ is as in the definition of local flatness (\Sr{loc flat}).

The following theorem of Carmesin \cite{CarEmbII} characterises all $2$-complex embeddable onto $\R^3$. This theorem is stated for simplicial complexes in \cite{CarEmbII}, but it is easy to extend to general $2$-complexes by
considering the barycentric subdivision, noting that a rotation system for the triangulated simplicial $2$-complex is canonically determined from the rotation system of the original $2$-complex.

\begin{theorem}[\cite{CarEmbII}] \label{JCtheorem}
A simply connected simplicial complex admits an embedding in $\R^3$ if and only if it admits a planar rotation system.
\end{theorem}

We will only use the easy forward direction of \Tr{JCtheorem}; the backward direction relies on Perelman's theorem.

\subsection{Uniqueness of rotation systems} \label{sec Urs}
The following lemma will help us  show that for any two embeddings $\phi, \psi$ of a  Whitney complex, there is a bijections between the $\phi$-chambers and the $\psi$-chambers. 

\begin{lemma}\label{lem: Sigma}
Let $X$ be a simply-connected, Whitney, $2$-complex. Suppose that $\phi, \psi : X \rightarrow \BS^3$ are two embeddings, and let  $\sigma$ and $\rho$ denote the  planar rotation systems induced by $\phi$ and $\psi$.
Then $\sigma$ and $\rho$ coincide up to  reversion.
\end{lemma}
\begin{proof}
Recall, from \Sr{sec Wh comp}, that $Y(X)$ is  the set of vertices $v\in X^0$ \st\ $L_X(v)$ contains a vertex of degree at least 3, in other words, an edge of $G(X)$, 
 and that $G(X)$ forms a connected graph with vertex set $Y(X)$.

If an edge $e\in X^1$ is not in $G(X)$, then it is incident with only two faces, hence $\sigma_e$  and $\rho_e$ is the unique cyclic ordering on two elements. In particular $\sigma$ and $\rho$ coincide on the complement of $G(X)$. 

To see that $\sigma$ and $\rho$ coincide on $G(X)$ as well, fix a vertex $v\in Y(X)$ (if $Y(X)$ is empty, then so is $G(X)$, and we are done by the above remark).
Then \Tr{Whitney} says that the rotation systems $\rho^v, \sigma^v$ on $L_X(v)$ induced by $\phi$ and $\psi$, respectively, coincide up to reversion. 
Pick $vw \in E(G(X))$, and recall that $\rho_{\overrightarrow{vw}} = \rho_{\overrightarrow{wv}}^{-1}$. Moreover, as $vw$ belongs to at least three cells of $X$, we have $(\rho_{\overrightarrow{wv}})^{-1}\neq \rho_{\overrightarrow{wv}}$. 
Recall that $\rho^{w}_v$ is determined by $\rho_{\overrightarrow{wv}}$.
As $vw$ belongs to at least three cells of $X$, if we know which of the two possible rotation systems on $L_X(v)$ is induced by $\psi$, then we can deduce which of the two possible rotation systems $\psi$ induces on $L_X(w)$. As $G(X)$ is a connected graph with vertex set $Y(X)$, by recursively repeating this argument we can determine the rotation system $\rho^u$ for all vertices $u\in Y(X)$ as well as the rotation system $\rho_e$ for all edges $e\in X^1$ with $e\subseteq Y(X)$ starting with $\rho^v$, and the same applies to $\sigma_e$ and $\sigma^v$.

For all edges $e\in X^1$ incident to a vertex not in $Y(X)$, only two faces are incident to $e$, hence the choice $\rho_e=\sigma_e$ is unique as already mentioned. To summarize, we have shown that $\rho^v$ determines the full rotation system $(\rho_e)_{e\in X^1}$. Moreover, as $\rho^v$ is either $\sigma^v$ or its reversion, we deduce that the rotation system $(\rho_e)_{e\in X^1}$ coincides with either $(\sigma_e)_{e\in X^1}$ or its reversion.
\end{proof}

\subsection{Local surfaces} \label{sec loc sur}

Assume a rotation system  $\sigma=(\sigma_e)_{e \in \overleftrightarrow{E}(X)}$ of a $2$-complex $X$ is induced from an embedding $\phi : X\rightarrow \BS^3$.
Then two faces $f,f'$ adjacent in $\sigma_e$ must lie on the boundary of the same {chamber}. This information can be encoded by using the notion of local surfaces defined in \cite{CarEmbII}.

Recall that each face of $X$ gives rise to two directed faces.
We say that two directed faces $f,f'$ of $X$ are \emph{locally related} via an edge $e=uv \in X^1$, with respect to the given planar rotation system $\sigma$, if $f$ appears right before $f'$ in $\sigma_{\overrightarrow{uv}}$, and $u$ appears right before $v$ in $f$, and $v$ appears right before $u$ in $f'$. 
A \emph{local surface} of $X$ with respect to $\sigma$ is an equivalence class of the transitive reflexive closure of the relation of being locally related. 

For a local surface $S$, a directed face $f\in S$, and a vertex $v \in f$, we define the \emph{$S$-local-disk}  at $v$ containing $f$, to be the set of directed faces $(f=)f_1,\dots, f_s$ each of which contains $v$, \st\  $f_i$ and $f_{i+1\pmod s}$ are locally related via an edge containing $v$ for each $i\in [s]$. For a directed face $f\in S$, we define \emph{the $S$-local-disk at $f$} to be the union of all $S$-local-disks (at the vertices of $f$) containing $f$. Note that the faces of a local-disk do not always form a topological disk ---although they do so in many cases. To see this, consider two copies of the tetrahedron in $\R^3$, glue them along an edge,  and let $S$ be the collection of exterior directed faces. Then it is easy to see that the local disk at a glued vertex or at any face is not a topological disk.




Suppose that $\phi$ is a locally flat embedding of $X$ into $\BS^3$, and consider a face $f$. Then any small enough neighbourhood $U_{\phi(y)}$ of any $y\in \phi(f)$ is separated into two components $W,W'$ by $\phi(f)$.  We would like to associate one of these components to each of the two possible directions $f_1,f_2$ of $f$ in a consistent way. This is not hard to do, e.g.\ by applying the `rule of thumb'. We then say that $f_1$ \defi{touches} $W$ and $f_2$ touches $W'$. We can then generalise the notion to define what it means for a directed face to touch a chamber or other region as follows. Given a connected open set $U\subset \BS^3$ (e.g.\ a chamber), we say that the directed face $f_1$ \defi{touches} $U$, if $\phi(f)\cap \partial U\neq \emptyset$ and for any $y\in \phi(f)\cap \partial U$ which is an interior point of $\phi(f)$, $f_1$ touches a component $W$ of $U_{\phi(y)} \sm \phi(f)$ contained in $U$.

Note that any directed face touches a unique chamber, and it is easy to see that two locally related directed faces touch the same chamber, hence the faces in a local surface touch a unique chamber.

\comment{
Consider a face $f= [x_1,\dots,x_t]$.
As $\phi$ is locally flat, for each $y\in \phi(f)$, we can find a collection of internally-disjoint curves $L_1,\dots, L_t$ on $\phi(f)$, where $L_i$ is a curve between $y$ and $\phi(x_i)$, and $L_i\cap L_j =\{y\}$ for $i\neq j\in [t]$.
As $\phi$ is locally flat, there exists an open neighbourhood $U_{\phi(y)}$ containing $\phi(y)$ such that $(U_{\phi(y)},\phi(f)\cap U_{\phi(y)})$ is homeomorphic to $(\mathbb{R}^3,\mathbb{R}^2)$. 
Hence,  $U_{\phi(y)}\cap \phi(f)$ divides $U_{\phi(y)}$ into two regions.  
A directed face $\langle x_1,\dots, x_t \rangle$ induces a cyclic ordering on the curves $L_1,\dots, L_t$, and this ordering determines a component $W(y,f)$ of $U_{\phi(y)}$ by the `rule of thumb'. 
We say that a directed face $f$ \defi{touches} a connected open set $U\subset \BS^3$, if $\phi(f)\cap \partial U\neq \emptyset$ and for any $y\in \phi(f)\cap \partial U$ which is an interior point of $\phi(f)$, the set $W(y,f)$ intersects  $U$. Note that any directed face touches a unique chamber, and it is easy to see that two locally related directed faces touch the same chamber, hence the faces in a local surface touch a unique chamber. 
}

\section{Boundaries of chambers} \label{sec chambers}

The rough idea of our proof of Theorem~\ref{Whit3D} is as follows. For  two given locally flat embeddings $\phi,\psi: X \rightarrow \BS^3$, we define a $2$-dimensional (simplicial) complex ${\rm fat}(X)$ containing $X$ and embeddings $\phi',\psi' : {\rm fat}(X)\rightarrow \BS^3$ which induces $\phi,\psi$ when restricted to $X$. 
We will show that any two embeddings $\phi', \psi'$ induce the same rotation systems on ${\rm fat}(X)$ (up to inversion) and this will define local surfaces of ${\rm fat}(X)$. Moreover, we will define ${\rm fat}(X)$ in a way that the closure of all chambers of ${\rm fat}(X)$ are homeomorphic to closed balls. By establishing bijections between chambers and local surfaces of ${\rm fat}(X)$, we find bijections between $\phi'$-chambers and $\psi'$-chambers. As any closure of such a chamber is homeomorphic to a closed ball, we can find homeomorphisms between closures of corresponding chambers which will yield the desired homeomorphism $\alpha$ once restricted to $X$.

The following three lemmas will help us to show that the $\phi'$-images and $\psi'$-images of local surfaces will be homeomorphic to $\BS^2$, and the closures of chambers are homeomorphic to closed balls. A \defi{surface} is a 2-dimensional manifold (without boundary).

\begin{lemma} \label{lem S2}
Let $X$ be a locally finite, simply-connected, 2-complex. Let $\phi : X \to M$ be an accumulation-free embedding in $M=\BS^3$ or $M=\R^3$ with bounded chambers, and let $C$ be a $\phi$-chamber. If $\partial C$ is  a surface, then $\partial C$ is homeomorphic to $\BS^2$.
\end{lemma}
\begin{proof}
Let $Z\subset M$ be the union of $X$ with all chambers of $\phi$ except $C$. Then $Z$ is homeomorphic to a solid surface embedded in $M$ with boundary homeomorphic to $S:= \partial C$. It is easy to see that $Z$ is simply-connected, because any loop in $Z$ is homotopic to a loop in $X$. Moreover, $S$ is connected by \Prr{pC conn}. Thus $S$ is homeomorphic with $\BS^2$.
\COMMENT{This may justify this: https://math.stackexchange.com/questions/1225617/is-a-compact-simply-connected-3-manifold-necessarily-s3-with-b3s-removed.

or this: https://math.stackexchange.com/questions/1286353/if-m-is-a-4-dimensional-compact-simply-connected-manifold-with-boundary-wha}

\end{proof}

In order to use Lemma~\ref{lem S2}, one has to show that the boundary of each chamber is a surface.
The boundary of a chamber will be the image of a local surface. Hence, we need to ensure that any local surface $S$ is a $2$-dimensional manifold. For this, we need to show that it is locally a disk, i.e.\ any $S$-local-disk at a vertex $v$ is homeomorphic to a closed disk and there is only one $S$-local-disk at $v$. The following lemma ensures the former. The uniqueness of the $S$-local-disk at $v$ is not always satisfied, and this is the reason we introduce the fattening ${\rm fat}(X)$ in the next section.

\begin{lemma}\label{lem: disk is disk}
Let $X$ be a locally $2$-connected simplicial $2$-complex, let  
$\sigma$ be a planar rotation system on $X$, and $S$ a local surface of $X$ with respect to $\sigma$.
Then for each vertex $v\in S^0$,  any $S$-local-disk $D$ at $v$ is homeomorphic to a closed disk. \COMMENT{this is one of the main obstacles for generalising to higher dimensions.}
\end{lemma}
\begin{proof}
Assume that we have a $S$-local-disk $D$ which consists of directed faces $f_1,\dots, f_{s+1}$ and for each $i\in [s]$, the face $f_i$ is locally related to $f_{i+1}$ via the edge $vu_i$. We claim that $u_1,\dots, u_s$ are all distinct.

Suppose they are not all distinct.
By shifting indices, we may assume that $u_1= u_{t}$ and $u_1,\dots, u_{t-1}$ are all distinct with $t\leq s$.
Then in the link graph $L_X(v)$, the vertices $u_1,\dots, u_{t-1}$ form a cycle $C$. Furthermore, the rotation system $\sigma^v$ induced by $\sigma$ is planar, so $L_X(v)$ has a planar embedding $\phi$ which induces $\sigma^v$.  
We may assume that $u_{t+1}$ lies outside $C$ in the planar embedding $\phi$ and  $u_i u_{i+1}$ is right after $u_{i-1}u_i$ in the cyclic ordering $\sigma^v_{u_i}$ for each $i\in [t-1]$.

Hence for each $i\in [t-1]\setminus\{1\}$, all neighbors of $u_i$ lie inside $C$ in the planar embedding $\phi$. However, this implies that $u_1=u_t$ is a cut vertex for $L_X(v)$, a contradiction as $X$ is locally $2$-connected. Hence $u_1,\dots, u_s$ are all distinct and $f_{s+1}=f_1$. As every face consists of three vertices, $\phi(D)$ is a homeomorphic to a closed disk.
\end{proof}

\comment{
\subsection{Move this stuff}
The following is a generalisation of the Jordan separation theorem to higher dimensions. \mymargin{Remove!} We will only need the $n=3$ case.

\begin{theorem}[\cite{Schmaltz}]  \label{Schmaltz}
Any  compact,  connected  hypersurface $X$ in $\R^n$  divides $\R^n$ into  two  connected  regions;  the  \defi{outside} $D_0$ and the \defi{inside} $D_1$.  Furthermore, the closure of $D_1$ is itself a compact manifold with boundary $X$.
\end{theorem}

\begin{corollary} \label{CorSchm}
Let $S$ be a closed surface and $\phi,\psi: S \to \R^3$ two embeddings of $S$. Then the interior of $\phi(S)$ is homeomorphic to the interior of $\psi(S)$.
\end{corollary}
}

\section{Fattenings} \label{sec fat}

We are trying to prove that  for every two locally flat embeddings $\phi,\psi: X \to \BS^3$, there exists a homeomorphism $\alpha : \BS^3 \to \BS^3$ such that $\psi= \alpha \circ \phi$. Our strategy is to build $\alpha$ as a combination of homeomorphisms, each mapping a chamber of $\phi$ to one of $\psi$. Since chambers are bounded by local surfaces, our task would be easier if we knew that each local surface is homeomorphic to a sphere. This is however not always the case, as a local surface may intersect itself along some edges or vertices of $X$ as remarked after its definition, see \fig{fig fat}. To avoid this difficulty we will now extend $X$ into a super-complex ${\rm fat}(X)$ called a \defi{fattening}. The idea is to glue a layer of prism-shaped `bricks' along the boundary of each chamber $C$ of $X$, with one brick being attached onto each directed face of $\partial C$. The sides of the bricks that are not contained in $\partial C$ or in another brick will then always form a surface, bounding a chamber lying inside $C$. The interior of each such brick will also form a new chamber inside $C$. 

\begin{figure}[htbp]
\vspace*{-0mm}
\centering
\noindent
 \includegraphics[scale=0.4]{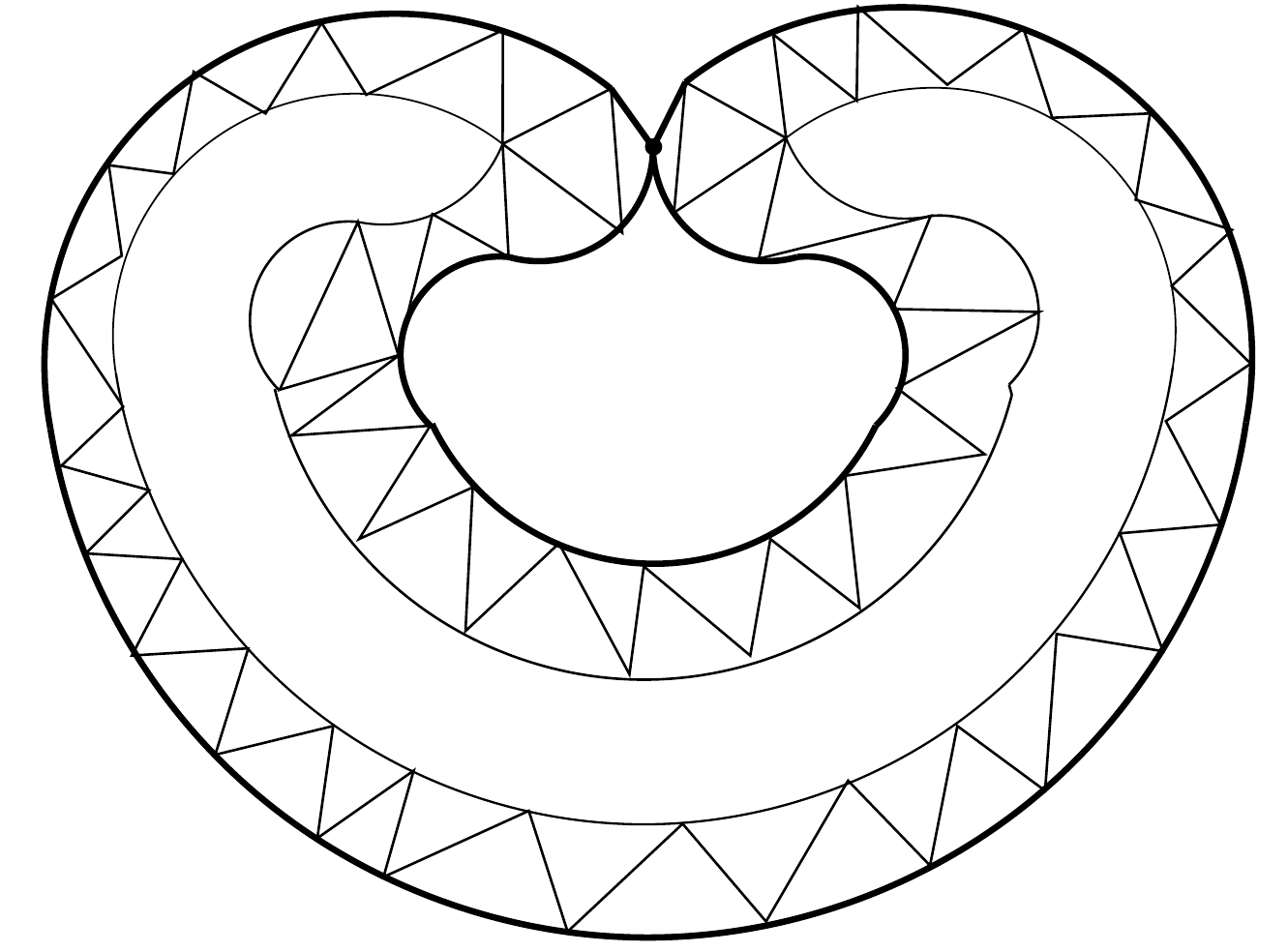}
\begin{minipage}[c]{0,95\textwidth}
\caption{A chamber $C$ may fail to be bounded by a surface, because e.g.\ its intersection with a plane in $\BS^3$ may look like the bold curve in the figure. But after glueing some prism-shaped bricks onto $\partial C$ (shaded area), we will ensure that the new chamber mirroring $C$ is  bounded by a surface (\Lr{rot fat}).}
\end{minipage}\label{fig fat} 
\end{figure}

The actual construction of ${\rm fat}(X)$ is however more technical than the above picture suggests, because rather than being given an embedding of $X$ defining its chambers, we will only be given a planar rotation system $\sigma$ on $X$, and we will construct ${\rm fat}(X)= {\rm fat}(X,\sigma)$ abstractly. Instead of chambers, we will have to work with the local surfaces and local disks of $X$ \wrt\ $\sigma$.

\subsection{Construction}

Now we formally define the fattening. Let $X$ be a locally $2$-connected simplicial $2$-complex, let $\sigma$ be a planar rotation system, and $S$ a local surface with respect to $\sigma$.
For each $S$-local-disk $D$ of $X$ \wrt\  $\sigma$ (at some vertex $w$), we introduce a new vertex $v_D$. For each directed face $f= \langle v_{1}, v_{2}, v_{3}\rangle   \in S$, we introduce a new face $h_f = \{ v_{D_1}, v_{D_2}, v_{D_3}\}  $ where $D_i$ is the $S$-local-disk at $v_i$ containing $f$. The 0-skeleton of ${\rm fat}(X)$ will be the union of $X^0$ with all the new vertices of the form $v_D$. Note that each vertex $w$ of $X$ gives rise to a number of new vertices of ${\rm fat}(X)$ equal to the number of $S$-local-disks containing $w$, which is at least one for each local surface containing $w$.

The new vertices and faces of ${\rm fat}(X)$ that we just introduced are not yet connected to $X$; they can be imagined as hovering inside the chambers of $X$, mirroring the boundaries thereof. However, it is important to remember that ${\rm fat}(X)$ is being defined \wrt\ an abstract rotation system rather than \wrt\ an embedding in $\BS^3$, and so the chambers of $X$ have to be imagined for now. 

We now connect the new vertices to $X$ by introducing new `rectangular' faces, each containing two new and two old vertices, so that each such rectangle has one side on an old face $f$ and the opposite side on $h_f$. We do so in such a way that  the space between each local surface of $X$ and its imagined mirror image is tiled by `bricks', with each brick having the shape of a prism between faces $f$ and $h_f$  (this picture will become more meaningful later when we define the embedding of ${\rm fat}(X)$).
These new faces are defined as follows. For each two directed faces $f$ and $h$ locally related  via an edge $uw$, we consider the  $S$-local-disk $D$ at $u$ containing $f$ and the  $S$-local-disk  $D'$ at $w$ containing $h$, so that $\{f,h\} \subseteq D\cap D'$. 
We introduce \mymargin{changed this} 
a new `rectangular' face $R_{f,h}:= [u, w, v_{D'}, v_{D}]$.


The resulting $2$-complex after all the above additions to $X$ is called a \emph{fattening} of $X$ with respect to $\sigma$, and we denote it by ${\rm fat}(X)={\rm fat}(X,\sigma)$. 

\subsection{Embedability of ${\rm fat}(X)$}

Our next aim is to prove that any locally flat embedding of $X$ can be extended to a locally flat embedding of ${\rm fat}(X)$. Notice that every Whitney complex is locally $2$-connected by the definitions.

\begin{lemma} \label{rot fat}
Let $\phi: X \to \BS^3$ be a locally flat embedding of a locally $2$-connected, simplicial, $2$-complex $X$,  and let $\sigma$ be the planar rotation system induced by $\phi$.
Then there is a locally flat embedding $\phi': {\rm fat}(X,\sigma) \rightarrow \BS^3$ that extends $\phi$. 

Moreover, letting $\sigma'$ be the planar rotation system induced by $\phi'$, every local surface of ${\rm fat}(X,\sigma)$ with respect to $\sigma'$ is a surface.
\end{lemma}

The proof is not difficult, and purely topological. Some readers will find it easy to prove \Lr{rot fat} by themselves, and such a reader can skip to \Sr{fat pres}. For all others we include the following proof.

In order to prove \Lr{rot fat} we will first show that there is enough space around the images of vertices, edges, and faces of $X$ where we can embed the newly introduced vertices, edges and faces of ${\rm fat}(X,\sigma)$. This is the purpose of the following technical lemma, which ensures that we can find some open balls disjoint from  $\phi(X)$ where we can embed those new cells. 

\begin{lemma}\label{lem: open ball}
Let $X$ be a locally $2$-connected, simplicial, $2$-complex, and $\phi$ a locally flat embedding of $X$ into $\BS^3$. Then for each $h\in X^0\cup X^1\cup X^2$, \ti\ a neighbourhood $N_h$ of $\phi(h)$ which is homeomorphic to an open ball and satisfies the following conditions.
For a local surface $S$, a vertex $v$, an edge $e$ containing $v$, two directed faces $f,g$ locally related via the edge $e$, and an $S$-local-disk $D$ at $f$, we have:
\begin{enumerate}
\item \label{N 0} $N_v\subseteq N_e\subseteq N_f$; 
\item \label{N i} for $h,h'\in X\cup X^1\cup X^2$ with $h\cap h'=\emptyset$, we have $N_h\cap N_{h'}=\emptyset$. If $h\cap h' = h''$ with $h''\in X^0\cup X^1$, then $N_h\cap N_{h'} = N_{h''}$;
\item \label{N ii} $\phi(D)\cap N_v$ divides $N_v$ into two regions. The region $W(v,f)$ that $f$ touches is homeomorphic to an open ball and disjoint from $\phi(X)$, and $W(v,f)=W(v,g)$;
\item \label{N iii}  $\phi(D)\cap N_e$ divides $N_e$ into two regions. The region $W(e,f)$ that $f$ touches is homeomorphic to an open ball and disjoint from $\phi(X)$, and $W(v,f)\subseteq W(e,f)$ and $W(e,f)=W(e,g)$; and
\item \label{N iv}  $\phi(D)\cap N_f$ divides $N_f$ into two regions.
 The region $W(f)$ that $f$ touches is homeomorphic to an open ball disjoint from $\phi(X)$ and $W(e,f)\subseteq W(f)$.
\end{enumerate}
\end{lemma}
\begin{proof}
As $\phi$ is a locally flat embedding, for each $x\in \phi(X)$, there exists an open neighbourhood $U_x$ satisfying (L1)--(L3) as in the definition of local flatness. We can shrink each $U_x$ if necessary to ensure that the open neighbourhood $N_f:=\bigcup_{x\in \phi(f)}U_x$ of $\phi(f)$ is homeomorphic to an open ball, and so is $N_e:= \bigcup_{x\in \phi(e)}U_x$. We just let $N_v = U_{\phi(v)}$. Then \ref{N 0} is satisfied.

By shrinking the $U_x$ further if necessary, we can easily also achieve that any non-empty intersection $N_h\cap N_{h'}$ is due to an intersection of $h$ and $h'$, establishing \ref{N i}.

As $N_v$ is an open ball satisfying (L3), the pair $(N_v,N_v\cap \phi(X))$ is homeomorphic to $(B_1,U_{L_X(v)})$ as in the definition of local flatness. In $(B_1,U_{L_X(v)})$, imagine the link graph $L_X(v)$ is drawn on the boundary of the open ball $B_1$ to induce $U_{L_X(v)}$.
By Lemma~\ref{lem: disk is disk}, the $S$-local-disk $D'$ at $v$ corresponds to a cycle in this planar drawing, which separates $\partial B_1$ into two regions. As the faces $f_1,\dots, f_s$ on $D'$ are locally related and $L_X(v)$ is $2$-connected, the region that $f$ touches contains no other points of $\phi(X)$ and is homeomorphic to an open ball. As $f,g$ are locally related, this is the region that $g$ touches as well.
Thus we have \ref{N ii}.

Suppose $e=uv$. Recall that $N_e$ is the union of $U_x$ for each $x\in \phi(e)$, and each $U_x$ satisfies (L2) and (L3). 
As (L2) and \ref{N ii} implies that $(U_x,U_x\cap \phi(D))$ is homeomorphic to $(\mathbb{R}^3,\mathbb{R}^2)$ for all $x\in \phi(e)$, we conclude that $N_e\cap \phi(D)$ divides $N_e$ into two regions. 
The region $W(e,f)$ that $f$ touches contains both $W(v,f)$ and $W(u,f)$. By our choice of $U_x$ for each $x\in {\rm int}(\phi(e))$, $W(e,f)$ does not intersect with $\phi(X)$ outside $W(v,f)\cup W(u,f)$. However, \ref{N ii} ensures each of $W(v,f)$ and $W(u,f)$ is homeomorphic to an open ball, hence we conclude that $W(e,f)$ is also homeomorphic to an open ball. 
\COMMENT{
For each $x\in {\rm int}(\phi(e))$, the set $U_{\phi(x)}\cap \phi(D)$ separates the open ball $U_{\phi(x)}$ into two regions by (L2).  Of these two regions we let $W(x,f)$ denote the one that $f$ touches. By (L2), this region does not intersect with $\phi(X)$.
Indeed, $W(e,f)$ is the union of $W(u,f)$ and $W(v,f)$ and $ W(x,f)$ for $x\in {\rm int}(\phi(e))$. 
These $W(x,f)$'s are all homeomorphic to an open ball as $U_x$ satisfies (L2). Hence $W(e,f)$ is also an open ball.} 
As $f$ and $g$ are locally related via $e$, $W(e,f)$ is also the region which $g$ touches.
By (L2), we know that $W(e,f)$ does not intersect with $\phi(X)$. Hence \ref{N iii} holds.

Finally, for each $x\in {\rm int}(\phi(f))$, the pair $(N_x, N_x\cap \phi(D))$ is homeomorphic to $(\mathbb{R}^3,\mathbb{R}^2)$. Hence, $N_x\cap \phi(D)$ divides $N_x$ into two regions, and the region $W(x,f)$ which $f$ touches is homeomorphic to an open ball. Thus $W(f) = \bigcup_{x\in \phi(f)} W(x,f)$ contains $W(e,f)$.  
Moreover, $N_f\cap \phi(D)$ is obtained from gluing $N_e\cap \phi(D)$ with $\phi(f)$ along with $\phi(f)\cap N_e$.
As $\phi(f)\cap N_e$ is a disk, by the Seifert-van Kampen theorem, $N_f\cap \phi(D)$ is a disk. Hence $N_f\cap \phi(D)$ divides $N_f$ into two regions and $W(f)$ is the region that $f$ touches. By \ref{N ii}--\ref{N iii} and (L1), $W(f)$ does not intersect with $\phi(X)$. Hence \ref{N iv} holds.
\end{proof}

We can now prove the main result of this section.

\begin{proof}[Proof of \Lr{rot fat}]
Let $\sigma$ be the planar rotation system induced by the embedding $\phi$.
Consider a local surface $S$ and directed face $f = \langle v_1,v_2,v_3\rangle  $ in $S$.
For each $i\in [3]$, let $g_i$ be the directed face locally related to $f$ via the edge $v_iv_{i+1}$, where here and below all subscripts are taken modulo three.
Let $D_{i}$ be the the $S$-local-disk at $v_{i}$  containing $f$.

By Lemma~\ref{lem: open ball}, we obtain open balls $W(v_i,f) \subseteq W(v_iv_{i+1},f) \subseteq W(f)$ satisfying (i)--(iv) of Lemma~\ref{lem: open ball}. 
As $W(f)$ is an open ball, we can embed a closed disk $D_f$ in $W(f)$ in the following way:
 its boundary is the union of three curves $C_{f,1}, C_{f,2}, C_{f,3}$ where $C_{f,i}$ lies in $W(v_iv_{i+1},f)= W(v_iv_{i+1},g_i)$ for each $i\in [3]$, and the point $C_{f,i}\cap C_{f,i+1}$ lies in $W(v_{i+1},f)$.  We can further ensure that all those disks $D_f$ are disjoint from each other for different $f$.

Consider two directed faces $f,g\in S$ locally related via an edge $uv$, and consider the $S$-local-disk $D$ at $u$ containing $f,g$, and the $S$-local-disk $D'$ at $v$ containing $f,g$.
Then we can continuously deform the disks to merge the corresponding boundaries  $C_{f,i}$ and $C_{g,j}$ (where $i,j$ are indices such that both $C_{f,i}$ and $C_{g,j}$ lie in $W(v_iv_{i+1},f)$) except the parts where they intersect $W(v_i,f)\cup W(v_{i+1},f)$.
For two faces $f,f'$, the sets $W(f)$ and $W(f')$ are disjoint if $f\cap f'=\emptyset$ by \ref{N i}, and $W(f)\cap W(f')=W(h,f)$ if $h$ is the cell consisting of the vertices in $f\cap f'$. Hence we can perform this deformation without two disks intersecting except at the merged boundaries.

Now for each vertex $v$, and each $S$-local-disk $D$ at $v$, let $D=\{f_1,\dots, f_s\}$ with $f_i$ locally related to $f_{i+1}$ via an edge $vu_i$.
Assume that a part of $C_{f_i,j}$ and a part of $C_{f_{i+1},j'}$ are merged. 
Then the not-yet-merged parts of $C_{f_i,j}$ and $C_{f_{i+1},j'}$ lie in $W(v,f_i)=W(v,f_{i+1})$ and $W(u,f_i)$. The parts in $W(v,f_i)$ all together form a topological circle in $\overline{W(v,f_i)}$.
Hence, we can continuously deform this circle into $s$ curves sharing one endpoint by merging the remaining part of $C_{f_i,j}$ and the remaining part of $C_{f_{i+1},j'}$. Furthermore, this will ensure that the point on the boundaries of $D_{f_1},\dots, D_{f_s}$ corresponding to $v$ are all identified.

By mapping each face $h_f$ as in the definition of ${\rm fat}(X)$ 
 into $D_{f}$ in such a way that each boundary edge of $h_f$ maps to $C_{f,1},C_{f,2}, C_{f,3}$, we obtain an embedding of all cells $h_f$ for each $f\in X$. Let $\phi''$ be the resulting embedding of all such disks $D_f$. Then $\phi \cup \phi''$ extends $\phi$ by embedding all the faces of the form $h_f$,  that mirror the original faces $f$ of $X$. It remains to embed the `rectangular' faces of the form $R_{f,h}$ 
that we used to connect the $h_f$ back to $X$. \mymargin{I removed some details here} This can be done similarly to the definition of $\phi''$ by introducing discs inside balls of the form 
$W(uv,f)$, using again appropriate deformations to stitch them together where needed. We omit the rather tedious, though straightforward, details.

\comment{
Moreover, for each locally related directed faces $f,g$ and $f\cap g =\{u,u'\}$ with a $S$-local-disk $D$ at $u$ and a $S$-local-disk $D'$ at $u'$ and $\{f,g\}\subseteq D\cap D'$, 
the image $\phi'(v_Dv_{D'})$ of the edge lies in $W(uu',f)$ which is homeomorphic to an open ball, and $\phi(uu')$ lies in the boundary of $W(uu',f)$.\COMMENT{As $W(uu',f)$ is the region we obtain from cutting $N_{uu'}$ by $\phi(D)\cup \phi(D')$.}

Let $D'_i$ be the $S$-local-disk at $u_i$ containing directed face $f_{i}$.
Note the $\phi'$-image of $\{h_{f_i} : i\in [s]\}$ is a disk and it divides $W(vu_i,f)$ into regions. As $\phi(vu_i)$ and $\phi'(v_Dv_{D'_i})$ lies in the closure of the same component, 
we can find another topological disk $D(f_i,f_{i+1})$ in the component
whose boundary consists of four curves
$\phi'(v_Dv_{D'_i})$ and $\phi(vu_i)$ and a curve $P_{f_i,f_{i+1}}(v)$ connecting $\phi(v)$ and $\phi'(v_D)$ and  a curve $P_{f_i,f_{i+1}}(u_i)$ connecting $\phi(u_i)$ and $\phi'(v_{D'_i})$.
Furthermore, we can ensure that all $D_{f_i,f_{i+1}}$ are disjoint except at the points $\phi(v)$ and $\phi'(v_{D})$.
This is possible as $W(vv_{i}, f_i)\cap W(vv_{i'},f_{i'})= W(v,f_i)$ and they are all open balls.

The curves $P_{f_i,f_{i+1}}(u)$ all lie in $W(u,f_i)$\COMMENT{For $i\neq i'\in [s]$, we have $W(u,f_i)=W(u,f_{i'})$}
and all of them shares the same ends $\phi(v)$ and $\phi'(v_D)$. So we can continuously move each $D(f_i,f_{i+1})$ to merge the boundary $P_{f_i,f_{i+1}}(u)$ to obtain one curve without introducing any intersection of the disks except the curves $P_{f_i,f_{i+1}}(u)$ merged.
We do this for each local surface $S$ then}
We thus obtain an embedding $\phi'$ of ${\rm fat}(X)$ which coincides with $\phi$ when restricted on $X$. 
This embedding is locally flat by construction.

\medskip 
We now prove the second part of our statement.
In order to check that a local surface $S'$ of ${\rm fat}(X)$ is a surface, it is sufficient to prove that there exists at most one $S'$-local-disk at $v$ for each vertex $v$. Indeed, Lemma~\ref{lem: disk is disk} implies that if there exists at most one $S'$-local-disk at $v$ for each vertex $v$, then $S'$ is locally a topological disk around each of its vertices, and so it is a $2$-manifold. 

We consider two types of local surfaces of ${\rm fat}(X)$, and show that each local surface belongs to one of these two types.
The first type will be denoted by $S_{\rm fat}$. For each local surface $S$ of $X$ with respect to $\sigma$, consider $S_{\rm fat}:= \{h_f: f\in S\}$ with $h_f$ directed in the same way as $f\in S$ is directed.
If $f,g$ are locally related, $h_f$ and $h_g$ are locally related too by the construction of $\phi'$. Hence, $S_{\rm fat}$ is indeed a local surface of ${\rm fat}(X)$ with respect to $\sigma'$.

The other type will be denoted by $T_f$.
For each directed face $f = \langle v_1,v_2,v_3 \rangle\in S$, let $D_i$ be the $S$-local-disk at $v_i$ containing $f$. There is exactly one directed face $g_i$ \st\ $\{f,g_i\} \subseteq D_i\cap D_{i+1} $ and $f,g_i$ are locally related via the edge $v_iv_{i+1}$.
Let $T_{f} := \{f, h_f^{-1}, R_{f,g_i} :  i\in [3]  \}$ where 
the rectangular faces $R_{f,g_i}$ are directed so that they are locally related, and each of them is locally related to $f$ and $h_f^{-1}$.
Then $T_{f}$ is a local surface of ${\rm fat}(X)$ with respect to $\sigma'$.

Each directed face of ${\rm fat}(X)$ belongs either to a local surface of type $S_{\rm fat}$ or a local surface of type $T_f$. Thus each local surface of ${\rm fat}(X)$ belongs to one of those types. 

We first verify our statement for a local surface of type $S_{\rm fat}$. 
First, note that each vertex in $S_{\rm fat}$ is of type $v_D$ for some $S$-local disk $D$. Moreover, as the  local-relatedness of $f$ and $g$ implies the local-relatedness of $h_f$ and $h_g$, the $S_{\rm fat}$-local disks are exactly $\{h_f: f\in D\}$ for some $S$-local disk $D$. Conversely, for each $S$-local disk $D$, the set $\{h_f: f\in D\}$ is an $S_{\rm fat}$-local disk. 
As we have constructed ${\rm fat}(X)$ so that $v_D$ and $v_{D'}$ are (even if $D$ and $D'$ are $S$-local disk at the same vertex) different vertices for two distinct $S$-local disks $D\neq D'$, we conclude that $\{h_f: f\in D\}$ is the unique $S_{\rm fat}$-local disk at $v_D$.
This shows that for every vertex $v_D$ of $S_{\rm fat}$, there exists exactly one $S_{\rm fat}$-local disk $\{h_f: f\in D\}$ at $v_D$ as desired.

We now consider a local surface of the second type $T_f$.
From the construction of $T_f$, it is easy to see that any two of its directed faces sharing an edge are locally related. Thus for each vertex $v$ in $T_{f}$, there exists a unique $T_{f}$-local disk at $v$, which is exactly the collection of directed faces in $T_{f}$ containing $v$.

We have checked that in all cases, there is exactly one $S'$-local disk at $v$ for any vertex $v$ of $S'$ as claimed.
\end{proof}

\comment{
	\subsection{Local surfaces of ${\rm fat}(X)$ are $2$-manifolds} \label{fat surf}

Next, we prove that ${\rm fat}(X)$ fulfils the purpose for which it was introduced: each of its local surfaces is a $2$-manifold:

\begin{lemma} \label{lem: g}
Suppose $\phi: X \to \BS^3$ is a locally flat embedding of a locally $2$-connected simplicial $2$-complex $X$. Let $\phi': {\rm fat}(X) \to \BS^3$ be a locally flat  embedding extending $\phi$. Let $\sigma'$ be the rotation  \mymargin{Instead of this lemma, can we please prove that every vertex $v$ has at most one $S$-local-disk as part of \Lr{rot fat}?} system induced by $\phi'$. Then every local surface $S$ of ${\rm fat}(X)$ with respect to $\sigma'$ is a $2$-manifold. 
\end{lemma}
\begin{proof}
Lemma~\ref{lem: disk is disk} implies that if there exists at most one $S$-local-disk at $v$ for each vertex $v$, then $S$ is locally a topological disk around each of its vertices, and so it is a $2$-manifold. Thus it is enough to show that every vertex $v$ has at most one $S$-local-disk.

Let $\sigma$ be the planar rotation system induced by $\phi$.
For each local disk $S'$ of $X$ with respect to $\sigma$, consider $S'_0:= \{h_f: f\in S\}$ with $h_f$ directed in the same way as $f\in S'$ is directed.
For each directed face $f = \langle v_1,v_2,v_3 \rangle\in S$, let $D_i$ be the $S$-local-disk at $v_i$ containing $f$,  and let $D_i\cap D_{i+1} \supseteq \{f,g_i\}$ where $f$ and $g_i$ are locally related via the edge $v_iv_{i+1}$.

It is easy to check that local surfaces of ${\rm fat}(X)$ with respect to $\sigma'$ is exactly the collection of all $S'_0$ and $T_f$ for all local surface $S'$ of $X$ and directed faces $f \in S'$. Also, by the definition of $S'_f, T_f$, each vertex appears in at most one of their local discs, as desired.
\end{proof}
}

\subsection{${\rm fat}(X)$ preserves the  Whitney property} \label{fat pres}

We conclude this section by proving that any fattening of a Whitney complex is itself Whitney.

\begin{lemma}\label{lem: 3-con}
Suppose that $X$ is a simply connected, Whitney, simplicial, $2$-complex,  $\phi: X \to \BS^3$ is a locally flat embedding, and $\sigma$ is the planar rotation system induced by $\phi$. Then every link graph of ${\rm fat}(X,\sigma)$ is a subdivision of a $3$-connected planar graph or a theta graph and every edge of ${\rm fat}(X,\sigma)$ belongs to at least three cells of ${\rm fat}(X,\sigma)$. In particular, ${\rm fat}(X,\sigma)$ is Whitney.
\end{lemma}
\begin{proof}

Since $\phi$ is locally flat, for each $v\in X^0$ there exists an open neighbourhood $U_v$ of $\phi(v)$ such that $(U_v, U_v\cap \phi(X))$ is homeomorphic to $(B_1,U_{L_X(v)})$, where $B_1$ and $U_{L_X(v)}$ are as defined in \Sr{loc flat}.
This yields an embedding of the link graph $L_X(v)$ on the boundary of $B_1$, which is $\BS^2$. 
Note that each $S$-local-disk $D$ at $v$ corresponds to a face $C_D$ in this planar embedding. 
Thus $L_{{\rm fat}(X)}(v)$ is obtained from $L_{X}(v)$ by adding a vertex $v_D$ for each $S$-local-disk $D$, and making $v_D$ adjacent to all vertices in $C_D$. 
Note that $L_X(v)$ contains no parallel edges, hence no 2-cycle, because $X$ is assumed to be a simplicial complex. It follows that if $L_{X}(v)$ is a subdivision of a $3$-connected planar graph, then so is $L_{{\rm fat}(X)}(v)$. Similarly, if $L_X(v)$ is a cycle or theta graph, then $L_{{\rm fat}(X)}(v)$ is a $3$-connected planar graph.

For a newly added vertex $v_D$, the link graph $L_{{\rm fat}(X)}(v_D)$ is a wheel, 
so it is also a $3$-connected planar graph. Thus every link graph of ${\rm fat}(X)$ is a subdivision of a $3$-connected planar graph as claimed.

Moreover, as $X$ is locally $2$-connected, every edge $uv\in X^1$ belongs to at least two faces and there are two local disks $D$ and $D'$ containing a face $f$. Thus $uv$ belongs to at least one more face in ${\rm fat}(X)$.
Moreover, it is straightforward to check that every newly added edge in ${\rm fat}(X)^1 \setminus X^1$ also belongs to at least three cells.
\end{proof}

{\bf Remark:} In the last two sections we never assumed $X$ to be finite. Thus all the results we obtained can be applied to embeddings in $\R^3$ instead of $\BS^3$, because $\BS^3$ is the 1-point compactification of $\R^3$ and so any (locally flat) embedding in $\R^3$ gives rise to a (locally flat)  embedding in $\BS^3$. 

\section{Putting the pieces together to prove \Trs{Whit3D} and \ref{Whit3D inf}} \label{sec proof}

We are ready to prove our main theorems.
The following lemma states that we can establish a homeomorphism between $\phi$ and $\psi$ as in \Tr{Whit3D} if all local surfaces are surfaces. Since ${\rm fat}(X)$ satisfies this condition by Lemma~\ref{rot fat}, \Tr{Whit3D} will follow easily. 

\begin{lemma}\label{lem: base case}
Suppose that $X$ is a finite, simply connected, Whitney, simplicial $2$-complex and $\phi, \psi: X \to \BS^3$ are two locally flat embeddings.
If the $\phi$-image and $\psi$-image of every local surface $S$ of $X$ with respect to $\sigma$ is
 a surface, then there exists a homeomorphism $\alpha: \BS^3 \to \BS^3$ such that $\phi = \alpha \circ \psi$.
\end{lemma}

We stated \Lr{lem: base case} for finite complexes embedded in $\BS^3$,  but the analogous statement for infinite, locally finite, complexes with accumulation-free embeddings in $\R^3$ with bounded chambers can be proved along the same lines using the remark at the end of the previous section, and it will be used in the proof of \Tr{Whit3D inf}.

\begin{proof}

Let $\sigma$ and $\rho$ be the planar rotation systems induced by $\phi$ and $\sigma$, respectively. 
By Lemma~\ref{lem: Sigma}, $\sigma$ and $\rho$ coinside up to reversion.
Hence, by taking a composition of an orientation-reversing automorphism with $\psi$, we may assume that $\phi$ and $\psi$ both induce the rotation system $\sigma$.

\begin{claim}\label{cl: sigma}
For each $\xi \in \{\phi,\psi\}$, there exists a bijection $i_{\xi}$ mapping $\xi$-chambers to local surfaces of $X$ with respect to $\sigma$ in such a way  that $\partial C = \xi(i_{\xi}(C))$ for every $\xi$-chamber $C$. 
\end{claim}
By \Prr{chambers},  $\partial C$ is a union of $ \xi$-images of faces. Let $A$ be the set of those directed faces $f$ \st\ $\xi(f) \subseteq \partial C$ and $f$ touches $C$  in $\xi$. As $\sigma$ is a rotation system induced by $\xi$, the definition of local-relatedness implies that if a directed face $f$ touches $C$ in $\xi$, then any directed face $f'$ locally related to $f$ also touches $C$ in $\xi$.  Hence, $A$ is a disjoint union of local surfaces.

As we are assuming each local surface to be a surface, $A$ is a surface, so
Lemma~\ref{lem S2} together with the simple connectedness of $X$ implies that $\xi(A)$ is homeomorphic to $\BS^2$.
As $\BS^2$ is not a disjoint union of two nonempty surfaces,  $A$ is a local surface. Hence we have $\partial C = \xi(S)$ for a local surface $S$, and we let $i_{\xi}(C):= S$.
As every directed face touches a chamber, every local surface is an image of $i_{\xi}$. Thus $i_{\xi}$ is the desired bijection, and Claim~\ref{cl: sigma} is proved. \newline

By Claim~\ref{cl: sigma}, the map $i_{\psi}^{-1}\circ i_{\phi}$ is a bijection from the $\phi$-chambers to the $\psi$-chambers.
Moreover, for a $\phi$-chamber $C_{\phi}$, let $C_{\psi} = (i_{\psi}^{-1}\circ i_{\phi})(C_{\phi})$. We deduce that for some local surface $S$ of $X$, $\partial C_{\phi} = \phi(S)$ and $\partial C_{\psi} = \psi(S)$ are both surfaces. 
By Lemma~\ref{lem S2}, both $\partial C_{\phi}$ and $\partial C_{\psi}$ 
are homeomorphic to $\BS^2$. Applying the generalized Schoenflies \Tr{GST} 
to each of $\partial C_{\phi}$ and $\partial C_{\psi}$ we obtain automorphisms $\alpha_\phi, \alpha_\psi$ of $\BS^3$ that map the equator to $\partial C_{\phi}$ and $\partial C_{\psi}$, respectively. Note that  $\alpha_\phi, \alpha_\psi$ must map one of the two 3-dimensional balls into which the equator separates $\BS^3$ onto $C_{\phi}, C_{\psi}$, respectively, and we may assume \obda\ that it is the same ball in both cases. Therefore, $\alpha_\psi \circ \alpha_\phi^{-1} $ restricts to a homeomorphism $\alpha_{C_{\phi}} : C_{\phi}\cup \partial C_{\phi} \to  C_{\psi}\cup \partial C_{\psi} $ extending the canonical homeomorphism 
$\psi \circ \phi^{-1}$ from $\partial C_{\phi}$ to $\partial C_{\psi}$.

Finally, we define $\alpha$ to be the map obtained by taking the union of all $\alpha_{C_{\phi}}$ for all $\phi$-chambers $C_{\phi}$.
For two different $\phi$-chambers $C$ and $C'$, the domains of $\alpha_{C}$ and $\alpha_{C'}$ are $C\cup \partial C$ and $C'\cup \partial C'$, respectively. 
As $C$ and $C'$ are disjoint, the intersection of these domains lies in $\partial C \cap \partial C' \subseteq \phi(X)$.
Since both $\alpha_{C}$ and $\alpha_{C'}$ extend $\psi \circ \phi^{-1}$, these two maps coincides on $\partial C\cap \partial C' \subseteq \phi(X)$. As this holds for any two distinct $\phi$-chambers $C$ and $C'$, this union $\alpha$ defines a homeomorphism from $\BS^3$ to $\BS^3$ with $\phi= \alpha\circ \psi$ as desired.
\end{proof}

We can now put the ingredients together to prove our main results:

\begin{proof}[Proof of  \Tr{Whit3D}]

\COMMENT{Recall that $G(X)$ is the subgraph of $X^1$ consisting of edges belonging to at least three faces and $Y(X)$ is the set of vertices $u$ whose link graph $L_X(u)$ is a subdivision of $3$-connected planar graph or a theta graph. As $X$ is Whitney, the graph $G(X)$ is connected, and it covers $Y(X)$.

 Take the barycentric subdivision on every face of $X$ to obtain a simplicial complex $X'$.  For each vertex $u \in X'^0 \setminus X^0$, it is easy to see that the link graph $L_{X'}(u)$ is a cycle. 
 For each $u\in X^0$, the link graph $L_{X'}(u)$ is the subdivision of the link graph $L_{X}(u)$ in such a way that every edge is subdivided once.
  Hence, $Y$ is exactly the set of the vertices $u$ whose link graph $L_{X'}(u)$ in $X'$ is a subdivision of planar $3$-connected graph or a theta graph. As $G(X)$ covers $Y(X)$, the complex $X'$ is also Whitney. 
}

By considering \mymargin{I've commented the two above paragraphs out; they were proving that the barycentric subdivision of a Whitney complex is Whitney, but this is now automatic from my new definition of a Whitney complex in the non-regular case. Please check this is fine.} 
the barycentric subdivision of $X$ (defined in \Sr{ccs}), we may assume that  $X$ is a simplicial complex, because barycentric subdivisions do not affect local flatness. 

Let $\sigma$ be the rotation system of $X$ induced by the embedding $\phi$ and let $\rho$ be the rotation system induced by the embedding $\psi$.
By Lemma~\ref{lem: Sigma}, $\sigma$ and $\rho$ coincide up to reversion.
By composing an orientation-reversing automorphism with $\psi$ if necessary, we may assume that $\sigma$ and $\rho$ coincide (without a reversion), and so from now on we will let $\sigma$ denote the common planar rotation system of $\phi$ and $\psi$.

This rotation system allows us to define local surfaces. Consider a fattening ${\rm fat}(X)$ of $X$ with respect to $\sigma$. 
 By Lemma~\ref{rot fat} there are  embeddings $\phi', \psi' : {\rm fat}(X)\rightarrow \BS^3$ such that each of them coincides with $\phi, \psi$ respectively when restricted to $X$.
 By  Lemma~\ref{rot fat} and Lemma~\ref{lem: 3-con}, ${\rm fat}(X)$ satisfies the assumptions of Lemma~\ref{lem: base case}.
Hence, there exists a homeomorphism $\alpha: \BS^3 \to \BS^3$ such that 
$\phi' = \alpha \circ \psi'$. 
As each of $\phi'$ and $\psi'$ coincide, when restricted to $X$, with
$\phi$ and $\psi$, respectively, this implies that $\phi=\alpha\circ \psi$ as desired.

\end{proof}

\begin{proof}[Proof of  \Tr{Whit3D inf}]
We follow the lines of the above proof. Lemmas~\ref{lem: Sigma}, \ref{rot fat} and~\ref{lem: 3-con} used above can be applied to embeddings in $\R^3$ instead of $\BS^3$ by using the remark at the end of the previous section. 
As mentioned above, Lemma~\ref{lem: base case} also readily adapts to accumulation-free embeddings in $\R^3$.
\end{proof}

\acknowledgement{We thank Johannes Carmesin for several helpful discussions.} 

\bibliographystyle{plain}
\bibliography{collective}

\begin{thebibliography}{10}

\bibitem{Alexander}
J.~W. Alexander.
\newblock An {Example} of a {Simply} {Connected} {Surface} {Bounding} a
  {Region} which is not {Simply} {Connected}.
\newblock {\em Proceedings of the National Academy of Sciences}, 10(1):8--10,
  1924.

\bibitem{BrownGS}
M.~Brown.
\newblock {A proof of the generalized Schoenflies theorem}.
\newblock {\em Bull. Amer. Math. Soc}, 66:74--76, 1960.

\bibitem{CarEmbItoV}
J.~Carmesin.
\newblock {Embedding simply connected 2-complexes in 3-space I-V}.
\newblock Preprints 2017.

\bibitem{CarEmbII}
J.~Carmesin.
\newblock {Embedding simply connected 2-complexes in 3-space II}.
\newblock Preprint 2017.

\bibitem{CarMihOut}
J.~Carmesin and T.~Mihaylov.
\newblock Outerspatial 2-complexes: {Extending} the class of outerplanar graphs
  to three dimensions.
\newblock {\em arXiv:2103.15404}, 2021.

\bibitem{diestelBook05}
Reinhard Diestel.
\newblock {\em Graph {T}heory \emph{(3rd edition)}}.
\newblock Springer-Verlag, 2005.
\newblock \\ Electronic edition available at:\\ {\small\tt
  http://www.math.uni-hamburg.de/home/diestel/books/graph.theory}.

\bibitem{FleUni}
H.~Fleischner.
\newblock The uniquely embeddable planar graphs.
\newblock {\em Discrete Mathematics}, 4(4):347--358, 1973.

\bibitem{Kleinian}
A.~Georgakopoulos.
\newblock {On planar Cayley graphs and Kleinian groups}.
\newblock {\em Trans.\ Am.\ Math.\ Soc.}, 373:4649--4684, 2020.

\bibitem{spheres}
A.~Georgakopoulos, J.~Haslegrave, R.~Montgomery, and B.~Narayanan.
\newblock {Spanning surfaces in 3-graphs}.
\newblock {\em {To appear in J.\ Eur.\ Math.\ Soc.}}

\bibitem{Hatcher}
A.~Hatcher.
\newblock {\em Algebraic {T}opology}.
\newblock Cambrigde Univ.\ Press, 2002.

\bibitem{ImWhi}
W.~Imrich.
\newblock {On Whitney's theorem on the unique embeddability of 3-connected
  planar graphs.}
\newblock In {\em {Recent Adv. Graph Theory, Proc. Symp. Prague 1974}}, pages
  303--306. 1975.

\bibitem{KLNSUni}
P.~Keevash, J.~Long, B.~Narayanan, and A.~Scott.
\newblock A universal exponent for homeomorphs.
\newblock {\em Israel Journal of Mathematics}, 243(1):141--154, 2021.

\bibitem{MazurGS}
B.~Mazur.
\newblock {On embeddings of spheres}.
\newblock {\em Bull. Amer. Math. Soc}, 65:59--65, 1959.

\bibitem{ThomassenRichter}
R.B. Richter and C.~Thomassen.
\newblock $3$-connected planar spaces uniquely embed in the sphere.
\newblock {\em Trans.\ Am.\ Math.\ Soc.}, 354:4585--4595, 2002.

\bibitem{whitney_congruent_1932}
H.~Whitney.
\newblock Congruent graphs and the connectivity of graphs.
\newblock {\em American J.\ of Mathematics}, 54(1):150--168, 1932.

\end{thebibliography}

\end{document}